\documentclass[11pt]{amsart}

\usepackage{a4wide}
\usepackage{amssymb}
\usepackage{mathrsfs}
\usepackage{enumerate}
\usepackage{esint}
\usepackage{titletoc}
\usepackage{mathscinet}
\usepackage[colorlinks=true, urlcolor=blue, linkcolor=blue, citecolor=green]{hyperref}
\usepackage[nameinlink]{cleveref}
\usepackage{mathtools}
\usepackage{xcolor}

\newtheorem{theorem}{Theorem}[section]
\newtheorem{lemma}[theorem]{Lemma}
\newtheorem{proposition}[theorem]{Proposition}
\newtheorem{corollary}[theorem]{Corollary}

\theoremstyle{plain}

\theoremstyle{definition}
\newtheorem{definition}[theorem]{Definition}

\theoremstyle{remark}
\newtheorem{remark}[theorem]{Remark}

\numberwithin{equation}{section} 

\makeatletter
\@namedef{subjclassname@2020}{\textup{2020} Mathematics Subject Classification}
\makeatother

\title[Quantitative bounds for H\"older exponents]{Quantitative bounds for H\"older exponents in the Krylov--Safonov and Evans--Krylov theories}

\author{Jongmyeong Kim}
\address{Institute of Mathematics, Academia Sinica, Taipei 106319, Taiwan}
\email{jmkim@gate.sinica.edu.tw}

\author{Se-Chan Lee}
\address{School of Mathematics, Korea Institute for Advanced Study, 02455 Seoul, Republic of Korea}
\email{sechan@kias.re.kr}


\thanks{S.-C. Lee is supported by the KIAS Individual Grant (No. MG099001) at the Korea Institute for Advanced Study.}

\begin{document}
\begin{abstract}
    We establish quantitative bounds for H\"older exponents in the Krylov--Safonov and Evans--Krylov theories when the ellipticity ratio is close to one. Our analysis relies on the Ishii--Lions method for the Krylov--Safonov theory and a Schauder-type perturbation argument for the Evans--Krylov theory.
\end{abstract}

\maketitle

\section{Introduction}
The regularity theory for second-order elliptic equations has attracted considerable attention and has been extensively developed in the literature. In the linear uniformly elliptic setting with measurable coefficients, De Giorgi--Nash--Moser~\cite{DG57, Nas58, Mos60, Mos61} established the H\"older continuity of weak solutions, providing a positive answer to Hilbert's nineteenth problem. On the other hand, the interior H\"older regularity of solutions in the non-divergence framework was achieved by Krylov--Safonov~\cite{KS79,KS80} by utilizing the ABP estimate and the Harnack inequality; see also Wang~\cite{Wan92a} for a related result in the parabolic setting. Since their approaches are essentially based on ellipticity rather than linearity itself, the Krylov--Safonov theory has been widely extended to the regularity theory for fully nonlinear uniformly elliptic equations; we refer to the monograph by Caffarelli--Cabre~\cite{CC95}.
	
To illustrate the issue, let $0<\lambda\leq \Lambda$ be ellipticity constants and let $\mathcal{S}^n$ be the space of real $n \times n$ symmetric matrices. We define \emph{Pucci's extremal operators} $\mathcal M^\pm_{\lambda,\Lambda}$ by
\begin{equation*}
	\mathcal{M}_{\lambda, \Lambda}^{+}(M)  = \sup_{\lambda I \leq A \leq \Lambda I} \text{tr} (AM)  \quad \text{and} \quad
	\mathcal{M}^{-}_{\lambda, \Lambda}(M)  =\inf_{\lambda I \leq A \leq \Lambda I} \text{tr} (AM).
\end{equation*}

The classical Krylov--Safonov theory can be stated as follows; see Section~\ref{sec-preliminary} for the definition of viscosity solutions.
\begin{theorem}[\cite{KS79,KS80}]\label{thm-KS}
    If $u\in C(B_1)$ is a viscosity solution of 
    \begin{equation}\label{eq-main}
    \mathcal M^+_{\lambda,\Lambda}(D^2u)\geq -K \quad \text{and} \quad \mathcal M^-_{\lambda,\Lambda}(D^2u)\leq K \quad \text{in $B_1$}
    \end{equation}
		 for some constant $K\ge 0$, then $u\in C^{\alpha}(\overline B_{1/2})$ for some $\alpha=\alpha(n,\lambda,\Lambda)\in(0,1)$.
\end{theorem}

It is noteworthy that Theorem~\ref{thm-KS} does not provide quantitative information on the universal H\"older exponent $\alpha(n,\lambda,\Lambda)$. The upper bound (in some sense) for the H\"older exponent $\alpha$ was captured in a series of papers by Nadirashvili--Vl{\u a}du{\c t}~\cite{NV07, NV08, NV13}. More precisely, they proved that: for every $\tau>0$, there exist a dimension $n$ and ellipticity constants $\lambda$ and $\Lambda$ such that there are solutions $u$ of \eqref{eq-main} with $u \notin C^{\tau}$. The counterexamples proposed in their works were constructed under the assumption that the ellipticity ratio is large enough, that is, $\Lambda/\lambda \gg 1$.

Recently, Lee--Yun~\cite{LY24} established the improved $C^{1, \alpha}$ regularity of viscosity solutions of \eqref{eq-main} when the ellipticity ratio is sufficiently close to one. 
    \begin{theorem}[\cite{LY24}]\label{thm-LY}
		Given $\alpha \in (0,1)$, there exists $\varepsilon=\varepsilon(\alpha,n)>0$ such that if $|\Lambda/\lambda-1|<\varepsilon$, then any viscosity solution $u\in C(B_1)$ of
		\eqref{eq-main} for some constant $K \geq 0$ belongs to $C^{1, \alpha}(\overline B_{1/2})$.
	\end{theorem}
The underlying idea of this result is that the Pucci extremal operators can be viewed as a small perturbation of the Laplacian operator $\Delta$ under the condition $\Lambda/\lambda \approx 1$. Nevertheless, since the proof of Theorem~\ref{thm-LY} strongly relies on the contradiction argument, the consequence in \cite{LY24} is still not quantitative in the sense that it does not provide an explicit relation between the H\"older exponent $\alpha$ and the ellipticity gap $\Lambda/\lambda-1$.
    
The first goal of this paper is to describe the dependence of the H\"older exponent, which appears in the Krylov--Safonov theory, on the ellipticity ratio in an explicit quantitative manner. 
\begin{theorem}[H\"older exponent bound in the Krylov--Safonov theory]\label{thm-holder}
	 Assume that $u \in C(B_1)$ is a viscosity solution of \eqref{eq-main} for some constant $K \geq 0$. Given $\alpha \in (0,1)$, if 
     \begin{equation*}
      \frac{\Lambda}{\lambda}-1 < 2\sqrt{\frac{1-\alpha}{n-1}},
    \end{equation*}
    then $u \in C^{\alpha}(\overline B_{1/2})$ with the uniform estimate
    \begin{equation*}
        \|u\|_{C^{\alpha}(\overline B_{1/2})} \leq C(\|u\|_{L^{\infty}(B_1)}+K),
    \end{equation*}
    where $C>0$ is a constant depending only on $n$, $\lambda$, $\Lambda$ and $\alpha$.
\end{theorem}

	A few remarks on Theorem~\ref{thm-holder} are in order. First, the criterion becomes stronger as $\alpha\to 1^-$, which is consistent with the perturbative intuition that one needs $\Lambda/\lambda\to 1^+$ for better regularity of solutions. Second, Mooney~\cite{Moo19} captured the behavior of the H\"older exponent: $\alpha \sim c(n)^{(\Lambda/\lambda)^{n-1}}$ for some $c(n) \in (0,1)$. While this asymptotic has an advantage that yields the information for general ellipticity ratio $\Lambda/\lambda$, it does not guarantee that $\alpha \to 1^-$ as $\Lambda/\lambda \to 1^+$. We also refer to Armstrong--Silvestre--Smart~\cite{ASS12}, Le~\cite{Le20} and Nascimento--Teixeira~\cite{NT23} for related results. Moreover, Theorem~\ref{thm-holder} has several simple but useful corollaries regarding fully nonlinear equations (Corollary~\ref{cor-fully}) and normalized $p$-Laplace equations (Corollary~\ref{cor-normalized-p}). Finally, our method can be applied to deal with both elliptic and parabolic equations that hold only where the gradient is large, whose regularity theory was developed by Imbert--Silvestre~\cite{IS16} in the elliptic setting. To the best of our knowledge, the H\"older regularity for viscosity solutions of the corresponding parabolic equations is new in this paper; see Theorem~\ref{thm-parabolic-space} and Corollary~\ref{cor-parabolic-time}.

	The proof of Theorem~\ref{thm-holder} is based on the Ishii--Lions method~\cite{IL90}, which has played a crucial role in developing H\"older or Lipschitz regularity of viscosity solutions in different contexts. For instance, we refer to the works by Capuzzo Dolcetta--Leoni--Porretta~\cite{CDLP10} for degenerate equations with superlinear Hamiltonians, Imbert--Silvestre~\cite{IS13} for degenerate fully nonlinear equations and Biswas--Topp~\cite{BT25} for fractional $p$-Laplace equations. In our setting, we consider a functional of the form
	\begin{equation*}
	    \Phi(x, y) \coloneqq u(x)-u(y)-L_1|x-y|^\alpha - L_2|x|^2 - L_2|y|^2,
	\end{equation*}
	and argue by contradiction. At a maximum point $(x,y) \in \overline B_{1/2}\times \overline B_{1/2}$ of this functional, we apply the Ishii--Jensen's lemma~\cite{Ish89, Jen88} to find appropriate semijets $(q_x,X)$ at $x$ and $(q_y,Y)$ at $y$ together with a matrix inequality for $X$ and $Y$. Then we exploit the matrix inequality to obtain the spectral information on $X$, $Y$ and $X-Y$ in terms of $L_1$, $L_2$ and $\alpha$. The comparison between the Pucci extremal operators and the Laplacian operator finally gives the desired quantitative relation between $\Lambda/\lambda-1$ and $\alpha$.\newline

    The second part of this paper is concerned with the Evans--Krylov theory~\cite{Eva82, Kry82, Kry83}. We say that a fully nonlinear operator $F : \mathcal{S}^n \to \mathbb{R}$ is
    \begin{enumerate}[(i)]
        \item \emph{$(\lambda, \Lambda)$-elliptic} if 
    \begin{equation*}
        \mathcal{M}^-_{\lambda, \Lambda}(M-N) \leq F(M)-F(N) \leq \mathcal{M}^+_{\lambda, \Lambda}(M-N) \quad \text{for any $M, N \in \mathcal{S}^n$};
    \end{equation*}

    \item \emph{concave} if
    \begin{equation*}
        F(\theta M+(1-\theta)N) \geq \theta F(M)+(1-\theta)F(N) \quad \text{for any $M, N \in \mathcal{S}^n$ and $\theta \in [0,1]$}.
    \end{equation*}
    \end{enumerate}
    
    The classical Evans--Krylov theory reads as follows. The key ingredient of the proof is the application of the weak Harnack inequality to derive the oscillation decay of $D^2u$. 
	\begin{theorem}[\cite{Eva82,Kry82,Kry83}]\label{thm-EK}
        Let $F : \mathcal{S}^n \to \mathbb{R}$ be uniformly $(\lambda,\Lambda)$-elliptic and concave.
	    If $u\in C(B_1)$ is a viscosity solution of $F(D^2u)=0$ in $B_1$, then $u\in C^{2,\alpha}(\overline B_{1/2})$ for some $\alpha=\alpha(n,\lambda,\Lambda)\in(0,1)$.
	\end{theorem}
	As in the Krylov--Safonov theory, Theorem~\ref{thm-EK} does not provide a quantitative information on the H\"older exponent $\alpha(n, \lambda, \Lambda)$. There are various attempts to replace the concavity assumption with a weaker one; we refer to Caffarelli--Yuan~\cite{CY00}, Savin~\cite{Sav07} and Huang~\cite{Hua19}, just to name a few. Recently, Wu--Niu~\cite{WN23} obtained an improved regularity of viscosity solutions of $F(D^2u)=0$ when the ellipticity ratio $\Lambda/\lambda$ is sufficiently close to one, without the concavity assumption on $F$.
    
	\begin{theorem}[\cite{WN23}]\label{thm-WN}
	    Let $F : \mathcal{S}^n \to \mathbb{R}$ be uniformly $(\lambda,\Lambda)$-elliptic. Given $\alpha\in(0,1)$, there exists $\varepsilon=\varepsilon(\alpha,n)>0$ such that if $|\Lambda/\lambda-1|<\varepsilon$, then any viscosity solution $u\in C(B_1)$ of $F(D^2u)=0$ in $B_1$ belongs to $C^{2,\alpha}(\overline B_{1/2})$. 
	\end{theorem}
		
	The proof of Theorem~\ref{thm-WN} also relies on the contradiction argument, and so it does not yield an explicit relation between $\alpha$ and $\Lambda/\lambda-1$. Our second goal is to provide an explicit quantitative criterion for the Evans--Krylov theory.
    
	\begin{theorem}[H\"older exponent bound in the Evans--Krylov theory]\label{thm-c2alpha-general}
   Let $F : \mathcal{S}^n \to \mathbb{R}$ be $(\lambda, \Lambda)$-elliptic and concave with $F(0)=0$. Moreover, let $u \in C(B_2)$ be a viscosity solution of
\begin{equation*}
    F(D^2u)=0 \quad \text{in $B_2$}.
\end{equation*}
There exists a universal constant $A_0>0$ independent of $n$, $\lambda$, $\Lambda$ and $\alpha$ such that: given $\alpha \in (0,1)$, if
\begin{equation*}
    \frac{\Lambda}{\lambda}-1 < \frac{1}{n} \left(\frac{1}{A_0n^3}\right)^{\frac{2+\alpha}{1-\alpha}},
\end{equation*}
then $u\in C^{2,\alpha}(\overline B_{1/2})$ with the uniform estimate
\begin{equation*}
    \|u\|_{C^{2,\alpha}(\overline B_{1/2})} \leq C\|u\|_{L^\infty(B_2)},
\end{equation*}
where $C>0$ depends only on $n$, $\lambda$, $\Lambda$ and $\alpha$.
\end{theorem}

Let us provide some remarks on Theorem~\ref{thm-c2alpha-general}. First, an explicit but non-optimized candidate for a universal constant $A_0$ can be found in Remark~\ref{rmk-constant-A}. Second, we can observe that the smallness requirement on $|\Lambda/\lambda-1|$ in Theorem~\ref{thm-c2alpha-general} is much stronger than the one in Theorem~\ref{thm-holder}. Finally, for the plane case (i.e., $n=2$), Goffi~\cite{Gof25} provided an explicit H\"older exponent $\alpha$ depending on $\Lambda/\lambda$. In fact, his result can be understood as the quantitative counterpart of the Nirenberg's result~\cite{Nir53} regarding $C^{2, \alpha}$ regularity of viscosity solutions of $F(D^2u)=0$ in $n=2$, without the concavity assumption on $F$.

	The proof of Theorem~\ref{thm-c2alpha-general} consists of several steps. We first regularize a fully nonlinear operator $F$ by considering its mollification $F^{\varepsilon}$. Then we are able to employ the Bernstein technique for the approximating solutions $u^{\varepsilon}$. At this stage, we derive interior $C^2$ estimate for $u^{\varepsilon}$ with tracking down the dependence of uniform constants on $n$, $\lambda$ and $\Lambda$. Finally, the interior estimate allows us to use the Schauder-type iteration argument based on the perturbation method. In other words, we show the approximation lemma that ensures the existence of a harmonic function close to $u^{\varepsilon}$, and then apply the standard iteration argument to upgrade such an approximation to the desired $C^{2, \alpha}$ estimate.\newline
	
	The paper is organized as follows. In Section~\ref{sec-preliminary}, we collect several definitions and preliminary results including viscosity solutions and the Ishii--Jensen's lemma. In Section~\ref{sec-KS}, we prove Theorem~\ref{thm-holder} by using the Ishii--Lions method and present several consequences for fully nonlinear equations, normalized $p$-Laplace equations and equations that hold only where the gradient is large. In Section~\ref{sec-EK}, we prove Theorem~\ref{thm-c2alpha-general} via a Schauder-type perturbation argument together with the interior estimates coming from the Bernstein technique.

\section{Preliminaries}\label{sec-preliminary}
We first display notations that will be used throughout the paper.
\begin{itemize}
    \item $B_r(x_0) \coloneqq \{x \in \mathbb{R}^n : |x-x_0|<r\}$. In particular, $B_1=B_1(0)$ is the unit ball in $\mathbb{R}^n$.

    \item $Q_r(x_0, t_0) \coloneqq B_r(x_0) \times [t_0-r^2, t_0)$.

    \item We say that $u \in C(B_1)$ belongs to the \emph{solution class} $S(\lambda, \Lambda, K)$ in $B_1$ for constants $0<\lambda \leq \Lambda$ and $K \geq 0$, if $u \in C(B_1)$ is a viscosity solution of \eqref{eq-main}.

    \item (Matrix norms) For $M \in \mathcal{S}^n$, $\|M\|$ denotes the $(L^2, L^2)$-norm of $M$, that is, 
    \begin{equation*}
        \|M\|=\sup_{|x|=1}|Mx|.
    \end{equation*}
    We note that
    \begin{equation*}
        \max_{i, j} |M_{ij}| \leq \|M\| \leq \|M\|_F \coloneqq \left(\sum_{i, j} M_{ij}^2\right)^{1/2} \leq \sqrt{n}\|M\|.
    \end{equation*}
\end{itemize}

\begin{definition}[Viscosity solutions]
	Let $u$ and $f$ be continuous functions defined in $B_1$.
	\begin{enumerate} [(i)]
		\item $u$ is called a \textit{viscosity subsolution} of $F(D^2u)=f$ in $B_1$ when the following condition holds: if for any $x_0 \in B_1$ and any smooth function $\varphi$ touching $u$ from above at $x_0$, then
		\begin{equation*}
			F(D^2\varphi(x_0)) \geq f(x_0).
		\end{equation*}
        
		\item $u$ is called a \textit{viscosity supersolution} of $F(D^2u)=f$ in $B_1$ when the following condition holds: if for any $x_0 \in B_1$ and any smooth function $\varphi$ touching $u$ from below at $x_0$, then
		\begin{equation*}
			F(D^2\varphi(x_0)) \leq f(x_0).
		\end{equation*}
	\end{enumerate}
    In a similar way, we can define viscosity solutions under more general situations (including parabolic ones); see \cite{CIL92, CC95} for details. 
\end{definition}

In order to use the Ishii--Lions method in Section~\ref{sec-KS}, we need to recall the Ishii-Jensen's lemma~\cite{Ish89, Jen88} that was originally developed to prove the uniqueness of viscosity solutions to fully nonlinear equations. Let us first display the elliptic version of Ishii--Jensen's lemma {\cite[Theorem~3.2]{CIL92}}. The definition of semijets $\mathcal{J}^{2, \pm}_{\mathcal{O}}$ and $\overline{\mathcal{J}}^{2, \pm}_{\mathcal{O}}$ can be found in \cite[Section~2]{CIL92}.
\begin{theorem}[Ishii-Jensen's lemma; elliptic]\label{lem-ij-elliptic}
    Let $\mathcal{O}_i$ be a locally compact subset of $\mathbb{R}^{N_i}$ for $i=1, \cdots, k$, 
    \begin{equation*}
        \mathcal{O}=\mathcal{O}_1 \times \cdots \times \mathcal{O}_k,
    \end{equation*}
    $u_i$ be an upper semicontinuous function in $\mathcal{O}_i$, and $\varphi$ be twice continuously differentiable in a neighborhood of $\mathcal{O}$. Set
    \begin{equation*}
        w(x)=u_1(x_1)+\cdots+u_k(x_k) \quad \text{for $x=(x_1, \cdots, x_k) \in \mathcal{O}$},
    \end{equation*}
 and suppose $\hat x=(\hat x_1, \cdots, \hat x_k) \in \mathcal{O}$ is a local maximum of $w-\varphi$ relative to $\mathcal{O}$.
 
Then for each $\varepsilon>0$, there exist $X_i \in \mathcal{S}^{N_i}$ such that
	\begin{enumerate}[(i)]
		\item $(D_{x_i}\varphi(\hat{x}), X_i) \in \overline{\mathcal{J}}^{2, +}_{\mathcal{O}_i}u_i(\hat{x}_i)$ for $i=1, \cdots, k$,
		\item 
		$
		-\left(\dfrac{1}{\varepsilon}+\|A\|\right)I \leq 
		\begin{pmatrix}
			X_1 & \cdots & 0 \\
			\vdots & \ddots & \vdots \\
			0 & \cdots & X_k
		\end{pmatrix}
		\leq A+\varepsilon A^2$,
	\end{enumerate}
	where $A=D^2\varphi(\hat x) \in \mathcal{S}^N$ for $N=N_1+\cdots+N_k$.
\end{theorem}

Moreover, the parabolic analogue of such a lemma can be found in \cite[Theorem~8.3]{CIL92} and \cite[Lemma~2.3.30]{IS13b}.
\begin{theorem}[Ishii-Jensen's lemma; parabolic]\label{lem-ij-parabolic}
Let $u_i$ be an upper semicontinuous function in $\mathcal{O}_i \times (0, T)$ for $i=1, \cdots, k$, where $\mathcal{O}_i$ is a locally compact subset of $\mathbb{R}^{N_i}$. Let $\varphi$ be defined on an open neighborhood of $\mathcal{O}_1 \times \cdots \times \mathcal{O}_k \times (0, T)$ and such that $(x_1, \cdots, x_k, t) \mapsto \varphi(x_1, \cdots, x_k, t)$ is once continuously differentiable in $t$ and twice continuously differentiable in $(x_1, \cdots, x_k) \in \mathcal{O}_1 \times \cdots \times \mathcal{O}_k$. Suppose that 
\begin{equation*}
    w(x_1, \cdots, x_k, t) \coloneqq u_1(x_1, t)+\cdots+u_k(x_k, t)-\varphi(x_1, \cdots, x_k, t)
\end{equation*}
attains a local maximum at $(\hat{x}_1, \cdots, \hat{x}_k, \hat{t}) \in \mathcal{O}_1\times \cdots \mathcal{O}_k \times (0, T)$. Assume, moreover, that there exists an $r>0$ such that for every $M>0$ there is a constant $C$ such that for $i=1, \cdots, k$,
 \begin{align*}
 	&b_i \leq C \quad \text{whenever} \quad (b_i, q_i, X_i) \in \mathcal{P}_{\mathcal{O}_i}^{2, +}u_i(x_i, t),\\
 	&|x_i-\hat{x}_i|+|t-\hat{t}|\leq r, \quad \text{and} \quad |u_i(x_i, t)|+|q_i|+\|X_i\| \leq M.
 \end{align*}
Then for each $\varepsilon>0$, there exist $X_i \in \mathcal{S}^{N_i}$ such that
	\begin{enumerate}[(i)]
		\item $(b_i, D_{x_i}\varphi(\hat{x}_1, \cdots \hat{x}_k, \hat{t}), X_i)) \in \overline{\mathcal{P}}_{\mathcal{O}_i}^{2, +}u_i(\hat{x}_i, \hat{t})$ for $i=1, \cdots, k$,
		\item 
		$
		-\left(\dfrac{1}{\varepsilon}+\|A\|\right)I \leq 
		\begin{pmatrix}
			X_1 & \cdots & 0 \\
			\vdots & \ddots & \vdots \\
			0 & \cdots & X_k
		\end{pmatrix}
		\leq A+\varepsilon A^2,$
		\item $b_1+\cdots+b_k=\varphi_t(\hat{x}_1, \cdots \hat{x}_k, \hat{t})$,
	\end{enumerate}
	where $A=(D_x^2\varphi)(\hat{x}_1, \cdots, \hat{x}_k, \hat{t})$.
\end{theorem}

We end this preliminary section with several well-known results, which will be used in Section~\ref{sec-EK}. The following interior estimates for the gradients of harmonic functions, which essentially come from the mean value property of harmonic functions together with an iteration argument. 

\begin{lemma}[{\cite[Theorem~2.10]{GT01}}]\label{lem-est-harmonic}
    Let $u$ be harmonic in $\Omega$ and let $\Omega'$ be any compact subset of $\Omega$. Then for any multi-index ${\theta}$ we have
    \begin{equation*}
        \sup_{\Omega'}|D^{\theta}u| \leq \left(\frac{n|\theta|}{d} \right)^{|\theta|} \sup_{\Omega}|u|,
    \end{equation*}
    where $d=\mathrm{dist}(\Omega', \partial \Omega)$.
\end{lemma}

We also need the stability theorem and the global H\"older regularity result when we derive the $C^{2, \alpha}$ regularity for general operators $F$ by passing the limit.
\begin{theorem}[{\cite[Proposition~2.9]{CC95}}]\label{thm-stability}
    Let $\{F_k\}_{k \geq 1}$ be a sequence of $(\lambda, \Lambda)$-elliptic operators, and let $\{u_k\}_{k \geq 1} \subset C(B_1)$ be such that $F_k(D^2u_k, x) \geq f(x)$ in the viscosity sense in $B_1$. Assume that $F_k$ converges uniformly in compact subsets of $\mathcal{S}^n \times B_1$ to $F$, and that $u_k$ converges uniformly in compact subsets of $B_1$ to $u$.

    Then $F(D^2u, x) \geq f(x)$ in the viscosity sense in $B_1$.
\end{theorem}

\begin{theorem}[{\cite[Proposition~4.14]{CC95}}]\label{thm-global-holder}
    Let $u \in C(\overline B_1)$ belong to the solution class $S(\lambda, \Lambda, K)$ in $B_1$. Let $\varphi \coloneqq u|_{\partial B_1}$ and let $\rho$ be a modulus of continuity of $\varphi$, that is, $\rho$ is a nondecreasing function in $(0, \infty)$ with $\lim_{\delta \to 0}\rho(\delta)=0$ such that
    \begin{equation*}
        |\varphi(x)-\varphi(y)| \leq \rho(|x-y|) \quad \text{for any $x, y \in \partial B_1$}.
    \end{equation*}
    Assume finally that $\|\varphi\|_{L^{\infty}(\partial B_1)} \leq K$.

    Then there exists a modulus of continuity $\rho^{\ast}$ of $u$ in $\overline B_1$, that is, $\rho^{\ast}$ is a nondecreasing function in $(0, \infty)$ with $\lim_{\delta \to 0}\rho^{\ast}(\delta)=0$ such that
    \begin{equation*}
        |u(x)-u(y)| \leq \rho^{\ast}(|x-y|) \quad \text{for any $x, y \in \overline B_1$},
    \end{equation*}
    and $\rho^{\ast}$ depends only on $n$, $\lambda$, $\Lambda$, $K$ and $\rho$.
\end{theorem}

\section{Quantitative H\"older exponent bound in Krylov--Safonov theory}\label{sec-KS}
Suppose that $u$ belongs to the solution class $S(\lambda, \Lambda, K)$ for some $K \geq 0$. The celebrated Krylov--Safonov theory says that $u \in C^{\alpha}(\overline B_{1/2})$ for \emph{some} $\alpha=\alpha(n, \lambda, \Lambda) \in (0,1)$. The goal of this section is to suggest a quantitative bound for the H\"older exponent $\alpha$ in terms of $n$ and $\Lambda/\lambda$, in different situations.

We begin with the quantitative H\"older exponent bound in the uniformly elliptic framework, whose proof is based on the Ishii--Lions method.
\begin{proof}[Proof of Theorem~\ref{thm-holder}]
    Without loss of generality, we may assume that $\|u\|_{L^{\infty}(B_1)}+K=1$ and we show the H\"older continuity of $u$ at the origin. For this purpose, we look for $L_1>0$ and $L_2>0$ such that
	\begin{equation}\label{eq-claim-elliptic}
		M=\sup_{x, y \in \overline B_{1/2}\times \overline B_{1/2}} \left[u(x)-u(y)-L_1 \varphi(|x-y|)-L_2|x|^2-L_2|y|^2\right] \leq 0,
	\end{equation}
	where $\varphi(s)=s^{\alpha}$ for given $\alpha \in (0,1)$. Indeed, if \eqref{eq-claim-elliptic} holds, then we set $y=0$ in \eqref{eq-claim-elliptic} to have
\begin{equation*}
u(x)-u(0)\leq L_1|x|^{\alpha}+L_2|x|^2\leq C|x|^{\alpha} \quad \text{for all } x\in B_{1/2}.
\end{equation*}
Similarly, by setting $x=0$ in \eqref{eq-claim-elliptic}, we have
\begin{equation*}
u(0)-u(y)\leq C|y|^{\alpha}\quad \text{for all } y\in B_{1/2}.
\end{equation*}
Hence, $u$ is $\alpha$-H\"older regular at the origin.
	
	We argue by contradiction: assume that $M>0$. If $(x, y) \in \overline B_{1/2} \times \overline B_{1/2}$ denotes a point where the maximum is attained, then we have
	\begin{equation*}
		L_1\varphi(|x-y|) +L_2|x|^2+L_2|y|^2 \leq |u(x)-u(y)| \leq 2.
	\end{equation*}
    In particular,
\begin{equation*}
|x|^2+|y|^2 \leq \frac{2}{L_2}
\end{equation*}
and
\begin{equation*}
\varphi(\theta)\leq \frac{2}{L_1}, \quad \text{where $\theta=|b|$ and $b=x-y$}.
\end{equation*}
At this stage, we fix $L_2>0$ large enough to ensure $x \neq y$ and $x, y \in B_{1/2}$.

We now apply the Ishii--Jensen's lemma (Theorem~\ref{lem-ij-elliptic}) to obtain that there exist a limiting subjet $(q_x, X)$ of $u$ at $x$ and a limiting superjet $(q_y, Y)$ of $u$ at $y$ such that
	\begin{align}\label{eq-matrixinequality}
		\begin{pmatrix}
		X & 0 \\
		0 & -Y
	\end{pmatrix}
        \leq L_1
        \begin{pmatrix}
		Z & -Z \\
		-Z & Z
	\end{pmatrix}
        +3L_2 I_{2n},
	\end{align}
    where
\begin{align*}
			q&\coloneqq L_1 \varphi'(\theta)\hat{b}, \quad q_x \coloneqq q+2L_2x, \quad q_y \coloneqq q-2L_2y,\\
			Z&\coloneqq\varphi''(\theta) \hat{b} \otimes \hat{b}+\frac{\varphi'(\theta)}{\theta} (I-\hat{b}\otimes \hat{b}) \quad \text{and} \quad \hat{b}\coloneqq\frac{b}{|b|}=\frac{x-y}{|x-y|}.
\end{align*}
First of all, due to the definition of semijets, we have
		\begin{equation*}
		\begin{cases}
			\mathcal{M}_{\lambda, \Lambda}^+(X) \geq -K & \quad \text{in $B_1$}\\
			\mathcal{M}_{\lambda, \Lambda}^-(Y) \leq K & \quad \text{in $B_1$}.
		\end{cases}
	\end{equation*}
Moreover, applying the matrix inequality \eqref{eq-matrixinequality} as a quadratic form inequality to vectors of the form $(\xi, \xi)$ for any $\xi \in \mathbb{R}^n$, we obtain
	\begin{equation*}
		\langle (X-Y)\xi, \xi \rangle \leq 6L_2 |\xi|^2.
	\end{equation*}
	Therefore, $X-Y \leq 6L_2I$ or equivalently, all eigenvalues of $X-Y$ are less than $6L_2$. On the other hand, considering the particular vector $(\hat b, -\hat b)$, we obtain
	\begin{equation*}
		\langle (X-Y)\hat b, \hat b\rangle \leq 6L_2+4L_1\varphi''(\theta).
	\end{equation*}
	Thus, at least one eigenvalue of $X-Y$ is less than $6L_2+4L_1\varphi''(\theta)$ (which will be a negative number). Since
    \begin{equation*}
        -Y \leq L_1 \varphi''(\theta) \hat b\otimes \hat b+L_1\frac{\varphi'(\theta)}{\theta}(I-\hat b \otimes \hat b)+3L_2I_n,
    \end{equation*}
    we obtain that at least one eigenvalue of $Y$ is larger than or equal to $-L_1\varphi''(\theta)-3L_2 (>0)$ and the other $(n-1)$-eigenvalues satisfy
    \begin{equation}\label{ineq1}
        e_i(Y) \geq -L_1\frac{\varphi'(\theta)}{\theta}-3L_2.
    \end{equation}
    In particular, we have
    \begin{equation*}
        K \geq \mathcal{M}_{\lambda, \Lambda}^-(Y)=\lambda \sum e_i^+(Y)-\Lambda \sum e_i^-(Y) 
    \end{equation*}
    or equivalently,
    \begin{equation}\label{ineq2}
        \sum e_i^+(Y) \leq \frac{\Lambda}{\lambda} \sum e_i^-(Y)+\frac{K}{\lambda}.
    \end{equation}
    The estimates \eqref{ineq1} and \eqref{ineq2} imply that
    \begin{align*}
      \mathrm{tr}\,Y=\sum  e_i^+(Y)-\sum e_i^-(Y) &\leq \left(\frac{\Lambda}{\lambda}-1\right) \sum e_i^-(Y)+\frac{K}{\lambda}\\
      &\leq \left(\frac{\Lambda}{\lambda}-1\right) (n-1)\left(L_1\frac{\varphi'(\theta)}{\theta}+3L_2 \right)+\frac{K}{\lambda}.
    \end{align*}
By combining all these estimates, it turns out that
\begin{equation*}
	\begin{aligned}
			\Lambda (n-1) \cdot 6L_2+ \lambda (6L_2+4L_1\varphi''(\theta)) &\geq \mathcal{M}_{\lambda, \Lambda}^+(X-Y) \\
			&\geq \mathcal{M}_{\lambda, \Lambda}^+(X) -	\mathcal{M}_{\lambda, \Lambda}^-(Y) +\mathcal{M}_{\lambda, \Lambda}^-(Y) -\mathcal{M}_{\lambda, \Lambda}^+(Y) \\
			&\geq -2K+\mathcal{M}_{\lambda, \Lambda}^-(Y) -\mathcal{M}_{\lambda, \Lambda}^+(Y).
	\end{aligned}
\end{equation*}
Here we observe that
\begin{equation*}
	\begin{aligned}
		\mathcal{M}_{\lambda, \Lambda}^+(Y) -\mathcal{M}_{\lambda, \Lambda}^-(Y) &=(\Lambda-\lambda)\mathrm{tr}Y\\
		&\leq (\Lambda-\lambda) \left(\left(\frac{\Lambda}{\lambda}-1\right) (n-1)\left(L_1\frac{\varphi'(\theta)}{\theta}+3L_2 \right)+\frac{K}{\lambda}\right).
	\end{aligned}	
\end{equation*}
Therefore, we arrive at
\begin{equation*}
\begin{aligned}
     &2K+6\Lambda (n-1)L_2+6\lambda L_2+(\Lambda-\lambda)\left(\left(\frac{\Lambda}{\lambda}-1\right) 3(n-1)L_2 +\frac{K}{\lambda}\right)  \\
     &\quad\geq L_1\left[-4\lambda \varphi''(\theta)-(\Lambda-\lambda) (n-1)\left(\frac{\Lambda}{\lambda}-1\right)\frac{\varphi'(\theta)}{\theta}\right].
\end{aligned}
\end{equation*}
We note that
\begin{equation*}
    -4\lambda \varphi''(\theta)-(\Lambda-\lambda) (n-1)\left(\frac{\Lambda}{\lambda}-1\right)\frac{\varphi'(\theta)}{\theta}>0
\end{equation*}
if and only if
\begin{equation*}
    \frac{-\theta\varphi''(\theta)}{\varphi'(\theta)} > \frac{n-1}{4}\left(\frac{\Lambda}{\lambda}-1\right)^2
\end{equation*}
if and only if
\begin{equation*}
    1-\alpha > \frac{n-1}{4}\left(\frac{\Lambda}{\lambda}-1\right)^2.
\end{equation*}
We now choose $L_1$ large enough to make a contradiction.
\end{proof}

The interior H\"older regularity of functions in the solution class can be used to establish the $C^{1, \alpha}$ regularity of solutions $u$ of the fully nonlinear equations.
\begin{corollary}[Fully nonlinear equation]\label{cor-fully}
    Let $F : \mathcal{S}^n \to \mathbb{R}$ be $(\lambda, \Lambda)$-elliptic and let $u \in C(B_1)$ be a viscosity solution of
    \begin{equation*}
        F(D^2u)=0 \quad \text{in $B_1$}.
    \end{equation*}
    Given $\alpha \in (0,1)$, if 
       \begin{equation*}
      \frac{\Lambda}{\lambda}-1 < 2\sqrt{\frac{1-\alpha}{n-1}},
    \end{equation*}
    then $u \in C^{1, \alpha}(\overline B_{1/2})$ with the uniform estimate
    \begin{equation*}
        \|u\|_{C^{1, \alpha}(\overline B_{1/2})} \leq C(\|u\|_{L^{\infty}(B_1)}+|F(0)|),
    \end{equation*}
    where $C>0$ is a constant depending only on $n$, $\lambda$, $\Lambda$ and $\alpha$.
\end{corollary}

\begin{proof}
    This is an immediate combination of Theorem~\ref{thm-holder} and \cite[Corollary~5.7]{CC95}.
\end{proof}

Moreover, Theorem~\ref{thm-holder} can be used to provide a quantitative counterpart of the result obtained in \cite[Theorem~1.4]{LY24}. In other words, we describe the relation between $|p-2|$ and $\alpha$ in a quantitative manner, when a viscosity solution $u$ of $\Delta_p^Nu=0$ belongs to $C^{1, \alpha}$ for $\alpha \in (0,1)$.

\begin{corollary}[Normalized $p$-Laplace equation]\label{cor-normalized-p}
    Let $u \in C(B_1)$ be a viscosity solution of
    \begin{equation*}
        \Delta_p^Nu=0 \quad \text{in $B_1$},
    \end{equation*}
    where $\Delta_p^N$ is the normalized $p$-Laplacian operator given by
    \begin{equation*}
        \Delta_p^Nu \coloneqq |Du|^{2-p}\Delta_pu=\left(\delta_{ij}+(p-2)\frac{u_iu_j}{|Du|^2}\right) u_{ij}.
    \end{equation*}
    Given $\alpha \in (0,1)$, if 
   \begin{equation*}
      \frac{\max\{p-1, 1\}}{\min\{p-1, 1\}}-1 <2 \sqrt{\frac{1-\alpha}{n-1}},
    \end{equation*}
    then $u \in C^{1, \alpha}(\overline B_{1/2})$ with the uniform estimate
    \begin{equation*}
        \|u\|_{C^{1, \alpha}(\overline B_{1/2})} \leq C\|u\|_{L^{\infty}(B_1)},
    \end{equation*}
    where $C>0$ is a constant depending only on $n$, $p$ and $\alpha$.
\end{corollary}

\begin{proof}
    We recall that
    \begin{equation*}
        \min\{p-1, 1\} I_n \leq  \delta_{ij}+(p-2)\frac{q_iq_j}{|q|^2} \leq \max\{p-1, 1\}I_n \quad \text{for all $q \in \mathbb{R}^n \setminus \{0\}$}.
    \end{equation*}
    In fact, to overcome a subtle issue arising from the singularity of $Du=0$, one has to introduce the regularized Dirichlet problem and to consider the approximated solution $u^{\varepsilon}$. Since the approximation method was already provided in \cite[Theorem~1.4]{LY24}, we omit the details here; a similar method will be used in Section~\ref{sec-EK}.    
\end{proof}

We next move our attention to the elliptic equations that hold only where the gradient is \emph{large}; see \cite{IS16} for details. For $\gamma \geq 0$, we let
\begin{equation*}
\begin{aligned}
     \mathcal{M}^+_{\lambda, \Lambda, *}(D^2u, \nabla u)&=\mathcal{M}^+_{*}(D^2u, \nabla u)\\
     &\coloneqq 
    \begin{cases}
        \Lambda \sum e_i^+(D^2u)+\lambda \sum e_i^-(D^2u)+\Lambda |\nabla u| & \text{if $|\nabla u| \geq \gamma$}\\
        +\infty & \text{otherwise},
    \end{cases}
\end{aligned}
\end{equation*}
\begin{equation*}
\begin{aligned}
    \mathcal{M}^-_{\lambda, \Lambda, *}(D^2u, \nabla u)&=\mathcal{M}^-_{*}(D^2u, \nabla u)\\
    &\coloneqq 
    \begin{cases}
        \lambda \sum e_i^+(D^2u)+\Lambda \sum e_i^-(D^2u)-\Lambda |\nabla u| & \text{if $|\nabla u| \geq \gamma$}\\
        -\infty & \text{otherwise}.
    \end{cases}
\end{aligned}
\end{equation*}
Imbert--Silvestre~\cite[Theorem~1.1]{IS16} proved that a viscosity solution $u \in C(B_1)$ of 
\begin{equation}\label{eq-elliptic-large}
        \mathcal{M}_{\ast}^+(D^2u, \nabla u) \geq -K \quad \text{and} \quad
			\mathcal{M}_{\ast}^-(D^2u, \nabla u) \leq K  \quad \text{in $B_1$}
	\end{equation}
is $\alpha$-H\"older continuous in $B_{1/2}$ for some small $\alpha \in (0,1)$. We note that if $u$ satisfies \eqref{eq-main}, then $u$ clearly satisfies \eqref{eq-elliptic-large}.

We again utilize the Ishii--Lions method to find a quantitative bound for $\alpha$.
\begin{theorem}[Elliptic; large gradient]\label{thm-large}
	 Assume that $u \in C(B_1)$ satisfies \eqref{eq-elliptic-large} for some constant $K \geq 0$. Given $\alpha \in (0,1)$, if 
     \begin{equation*}
      \frac{\Lambda}{\lambda}-1 < 2\sqrt{\frac{1-\alpha}{n-1}},
    \end{equation*}
    then $u \in C^{\alpha}(\overline B_{1/2})$ with the uniform estimate
    \begin{equation*}
        \|u\|_{C^{\alpha}(\overline B_{1/2})} \leq C(\|u\|_{L^{\infty}(B_1)}+K),
    \end{equation*}
    where $C>0$ is a constant depending only on $n$, $\lambda$, $\Lambda$, $\alpha$ and  $\gamma/(\|u\|_{L^{\infty}(B_1)}+K)$.
\end{theorem}

\begin{proof}
    As in the proof of Theorem~\ref{thm-holder}, we may assume that $\|u\|_{L^{\infty}(B_1)}+K=1$ and it is enough to find $L_1>0$ and $L_2>0$ such that
	\begin{equation*}
		M=\sup_{x, y \in \overline B_{1/2}\times \overline B_{1/2}} \left[u(x)-u(y)-L_1 \varphi(|x-y|)-L_2|x|^2-L_2|y|^2\right] \leq 0,
	\end{equation*}
	where $\varphi(s)=s^{\alpha}$ for given $\alpha \in (0,1)$. 
	
	We argue by contradiction: assume that $M>0$. If $(x, y) \in \overline B_{1/2} \times \overline B_{1/2}$ denotes a point where the maximum is attained, we have
	\begin{equation}\label{eq-max}
		L_1\varphi(|x-y|) +L_2|x|^2+L_2|y|^2 \leq |u(x)-u(y)| \leq 2.
	\end{equation}
	If we choose $L_2$ large enough, then $x, y \in B_{1/2}$ and $x \neq y$.
	
	We use the same notation as in the proof of Theorem~\ref{thm-holder}. In view of \eqref{eq-max}, we have $L_1\theta^{\alpha} \leq 1$ and so if we choose $L_1$ large enough, then $\theta$ will be small and $\varphi'(\theta)$ will be large. In particular, we arrive at $|q|=L_1\varphi'(\theta) \geq 2\max\{L_2, \gamma\}$ and so 
    \begin{equation*}
       \frac{|q|}{2} \leq |q_x|=|q+2L_2x| \leq 2|q| \quad \text{and} \quad \frac{|q|}{2} \leq |q_y|=|q-2L_2y| \leq 2|q|.
    \end{equation*}	
      Since $|q_x| \geq \gamma$ and $|q_y| \geq \gamma$, we have
		\begin{equation*}
		\begin{cases}
			\mathcal{M}_{\ast}^+(X, q_x)=\Lambda \sum e_i^+(X)+\lambda \sum e_i^-(X)+\Lambda |q_x| \geq -K & \quad \text{in $B_1$}\\
			\mathcal{M}_{\ast}^-(Y, q_y)=\lambda \sum e_i^+(Y)+\Lambda \sum e_i^-(Y)-\Lambda |q_y| \leq K & \quad \text{in $B_1$}.
		\end{cases}
	\end{equation*}
    Therefore, we repeat the argument in the proof of Theorem~\ref{thm-holder} to obtain
    \begin{equation*}
        K+2\Lambda |q| \geq \lambda \sum e_i^+(Y)-\Lambda \sum e_i^-(Y)
    \end{equation*}
    and so 
    \begin{equation*}
    \begin{aligned}
        \mathrm{tr}Y=\sum e_i^+(Y)-\sum e_i^-(Y) &\leq \left(\frac{\Lambda}{\lambda}-1\right) \sum e_i^-(Y)+\frac{K}{\lambda}+2\frac{\Lambda}{\lambda} |q|\\
        &\leq  \left(\frac{\Lambda}{\lambda}-1\right) (n-1)\left(L_1\frac{\varphi'(\theta)}{\theta}+3L_2\right)+\frac{K}{\lambda}+2\frac{\Lambda}{\lambda} |q|.
    \end{aligned}
    \end{equation*}
By combining these estimates, it turns out that
\begin{equation*}
	\begin{aligned}
			\Lambda (n-1) \cdot 6L_2&+\lambda (6L_2+4L_1\varphi''(\theta))\\
            &\geq \mathcal{M}_{\lambda, \Lambda}^+(X-Y) \\
			&\geq (\mathcal{M}_{\ast}^+(X, q_x)-\Lambda |q_x|) -(\mathcal{M}_{\ast}^-(Y, q_y)+\Lambda |q_y|) +\mathcal{M}_{\lambda, \Lambda}^-(Y) -\mathcal{M}_{\lambda, \Lambda}^+(Y)\\
			&\geq -2K-\Lambda |q_x|-\Lambda |q_y|+\mathcal{M}_{\lambda, \Lambda}^-(Y) -\mathcal{M}_{\lambda, \Lambda}^+(Y).
	\end{aligned}
\end{equation*}
Here we observe that
\begin{equation*}
	\begin{aligned}
		\mathcal{M}_{\lambda, \Lambda}^+(Y) -\mathcal{M}_{\lambda, \Lambda}^-(Y) &=(\Lambda-\lambda)\mathrm{tr}Y\\
		&\leq (\Lambda-\lambda) \left(\left(\frac{\Lambda}{\lambda}-1\right) (n-1)\left(L_1\frac{\varphi'(\theta)}{\theta}+3L_2\right)+\frac{K}{\lambda}+2\frac{\Lambda}{\lambda} |q|\right).
	\end{aligned}	
\end{equation*}
Therefore, by recalling that $|q|=L_1\varphi'(\theta)$, we arrive at
\begin{equation*}
\begin{aligned}
     &2K+6\Lambda (n-1)L_2+ 6\lambda L_2+(\Lambda-\lambda)\left(\left(\frac{\Lambda}{\lambda}-1\right)(n-1)L_2+\frac{K}{\lambda}\right)  \\
     &\quad\geq L_1\left[-4\lambda \varphi''(\theta)-(\Lambda-\lambda) (n-1)\left(\frac{\Lambda}{\lambda}-1\right)\frac{\varphi'(\theta)}{\theta}-4\Lambda \varphi'(\theta)-2(\Lambda-\lambda)\frac{\Lambda}{\lambda}\varphi'(\theta)\right].
\end{aligned}
\end{equation*}
Since $\theta \to 0$ as $L_1 \to \infty$, we may choose $L_1=L_1(\varepsilon, n, \lambda, \Lambda, \gamma)$ larger so that 
\begin{equation*}
    4\Lambda \varphi'(\theta)+2(\Lambda-\lambda)\frac{\Lambda}{\lambda}\varphi'(\theta) \leq -4\varepsilon \lambda \varphi''(\theta)
\end{equation*}
for given $\varepsilon \in (0,1)$. Then we note that
\begin{equation*}
    -4\lambda (1-\varepsilon)\varphi''(\theta)-(\Lambda-\lambda) (n-1)\left(\frac{\Lambda}{\lambda}-1\right)\frac{\varphi'(\theta)}{\theta}>0 \quad \text{for some $
\varepsilon \in (0,1)$}
\end{equation*}
if and only if
\begin{equation*}
    1-\alpha > \frac{n-1}{4}\left(\frac{\Lambda}{\lambda}-1\right)^2.
\end{equation*}
We now choose $L_1$ further large enough to make a contradiction.
\end{proof}

We finally develop the parabolic counterpart of Theorem~\ref{thm-large}, i.e., the H\"older regularity of viscosity solutions of \eqref{eq-parabolic-large} in both space and time variables. We would like to point out that the parabolic analogue of \cite{IS16} was open, to the best of the authors' knowledge. In other words, it was even unknown whether a viscosity solution of \eqref{eq-parabolic-large} is H\"older continuous or not. In the following theorem, we prove the H\"older continuity of $u$ in the space variable under the assumption that $\Lambda/\lambda$ is close to $1$.
\begin{theorem}[Parabolic; large gradient 1]\label{thm-parabolic-space}
    Assume that $u \in C(Q_1)$ satisfies
	\begin{equation}\label{eq-parabolic-large}
    \left\{
    \begin{aligned}
        u_t-\mathcal{M}_{\ast}^+(D^2u, \nabla u) &\leq K && \quad \text{in $Q_1$}\\
			u_t-\mathcal{M}_{\ast}^-(D^2u, \nabla u) &\geq -K && \quad \text{in $Q_1$}
    \end{aligned}
		\right.
	\end{equation}
    for some constant $K \geq 0$. Given $\alpha \in (0,1)$, if 
       \begin{equation*}
      \frac{\Lambda}{\lambda}-1 < 2\sqrt{\frac{1-\alpha}{n-1}},
    \end{equation*}
    then  $u(\cdot, t) \in C^{\alpha}_x$ with the uniform estimate
    \begin{equation*}
       |u(x, t)-u(y, t)| \leq C(\|u\|_{L^{\infty}(Q_1)}+K)|x-y|^{\alpha} \quad \text{for any $(x, t), (y, t) \in \overline Q_{1/2}$}, 
    \end{equation*}
    where $C>0$ is a constant depending only on $n$, $\lambda$, $\Lambda$, $\alpha$ and $\gamma/(\|u\|_{L^{\infty}(Q_1)}+K)$.
\end{theorem}

\begin{proof}
Without loss of generality, we may assume that $\|u\|_{L^{\infty}(Q_1)}+K=1$ and we show the H\"{o}lder continuity of $u$ at $(0,0)$. As in the elliptic case, it suffices to prove
\begin{equation}\label{claim-parabolic}
		M\coloneqq\max_{\substack{x, y \in \overline B_{1/2}\times \overline B_{1/2}  \\ t \in [-1/4, 0] }} \left[u(x, t)-u(y,t)-L_1\varphi(|x-y|)-L_2|x|^2-L_2|y|^2-L_2t^2 \right] \leq 0,
\end{equation}
where $\varphi(s) = s^{\alpha}$ for given $\alpha \in (0,1)$.

We prove \eqref{claim-parabolic} by contradiction: suppose that the positive maximum $M$ is attained at $t \in [-1/4, 0]$ and $x, y \in \overline{B}_{1/2}$. It immediately follows that $x \neq y$ and
\begin{equation*}
L_1\varphi(|x-y|)+L_2|x|^2+L_2|y|^2+L_2t^2 \leq 2\|u\|_{L^{\infty}(Q_1)}\leq 2.
\end{equation*}
In particular,
\begin{equation*}
|x|^2+|y|^2+|t|^2 \leq \frac{2}{L_2}
\end{equation*}
and
\begin{equation}\label{eq-theta}
\varphi(\theta)\leq \frac{2}{L_1}, \quad \text{where $\theta=|b|$ and $b=x-y$}.
\end{equation}
We fix $L_2>0$ large enough to ensure $t \in (-1/4, 0]$ and $x, y \in B_{1/2}$.
	
We now apply the parabolic version of Ishii--Jensen's lemma (Theorem~\ref{lem-ij-parabolic}) to obtain that there exist a limiting subjet $(\sigma_x, q_x, X)$ of $u$ at $(x, t)$ and a limiting superjet $(\sigma_y, q_y, Y)$ of $u$ at $(y, t)$ such that
\begin{itemize}
			\item
			$
			\begin{pmatrix}
				X & 0 \\
				0 & -Y
			\end{pmatrix}
			\leq L_1
			\begin{pmatrix}
				Z & -Z \\
				-Z & Z
			\end{pmatrix}
			+3L_2I_{2n}
			;$
			\item $\sigma_x-\sigma_y=2L_2t$,
\end{itemize}
where
\begin{align*}
			q&\coloneqq L_1 \varphi'(\theta)\hat{b}, \quad q_x \coloneqq q+2L_2x, \quad q_y \coloneqq q-2L_2y,\\
			Z&\coloneqq\varphi''(\theta) \hat{b} \otimes \hat{b}+\frac{\varphi'(\theta)}{\theta} (I-\hat{b}\otimes \hat{b}) \quad \text{and} \quad \hat{b}\coloneqq\frac{b}{|b|}=\frac{x-y}{|x-y|}.
\end{align*}
By choosing $L_1$ large enough, $\theta$ will be small, $|\varphi'(\theta)|$ and so $|q| \geq 2\max\{L_2, \gamma\}$. In particular, we have
\begin{equation*}
\frac{|q|}{2} \leq |q_x|=|q+2L_2x| \leq 2|q| \quad \text{and} \quad \frac{|q|}{2} \leq |q_y|=|q-2L_2y| \leq 2|q|.
\end{equation*}
Then it follows from the definition of viscosity solution that
\begin{equation*}
    2L_2t=\sigma_x-\sigma_y \leq 2K+\mathcal{M}_{\ast}^+(X, q_x) -	\mathcal{M}_{\ast}^-(Y, q_y).
\end{equation*}
Since the remaining part of the proof is identical to the one of Theorem~\ref{thm-large}, we omit the details.
\end{proof}

In fact, Theorem~\ref{thm-parabolic-space} was only concerned with the H\"older regularity of viscosity solutions of \eqref{eq-parabolic-large} in the space variable. We now follow the argument in \cite[Lemma~9.1]{BBL02} to derive the H\"older regularity of solutions in the time variable as follows. 

\begin{corollary}[Parabolic; large gradient 2]\label{cor-parabolic-time}
    Assume that $u \in C(Q_1)$ satisfies \eqref{eq-parabolic-large} for some constant $K \geq 0$. Given $\alpha \in (0,1)$, if 
    \begin{equation*}
      \frac{\Lambda}{\lambda}-1 < 2\sqrt{\frac{1-\alpha}{n-1}},
    \end{equation*}
    then $u(x, \cdot) \in C^{\alpha/2}_t$ with the uniform estimate
    \begin{equation*}
       |u(x, t_1)-u(x, t_2)| \leq C(\|u\|_{L^{\infty}(Q_1)}+K)|t_1-t_2|^{\frac{\alpha}{2}} \quad \text{for any $(x, t_1), (x, t_2) \in \overline Q_{1/2}$}, 
    \end{equation*}
    where $C>0$ is a constant depending only on $n$, $\lambda$, $\Lambda$, $\alpha$ and $\gamma/(\|u\|_{L^{\infty}(Q_1)}+K)$.
\end{corollary}

\begin{proof}
As usual, we may assume that $\|u\|_{L^{\infty}(Q_1)}+K=1$. We claim that, for any $\eta>0$ and $t_0 \in [-1/4, 0)$, we can find constants $L_1, L_2>0$ large enough such that 
\begin{equation}\label{eq-claim-holder-time}
    u(x, t)-u(0, t_0) \leq \eta+L_1(t-t_0)+L_2|x|^2 \eqqcolon \varphi(x, t) \quad \text{for every $(x, t) \in \overline B_{1/2} \times [t_0, 0]$}.
\end{equation}
We first choose $L_2 \geq 8$ so that \eqref{eq-claim-holder-time} holds for $x \in \partial B_{1/2}$. We next choose $L_2$ larger to ensure that \eqref{eq-claim-holder-time} holds for $t=t_0$. To this end, we apply Theorem~\ref{thm-parabolic-space} to find that 
\begin{equation*}
    u(x, t_0)-u(0, t_0) \leq [u]_{C_x^{\alpha}}|x|^{\alpha}.
\end{equation*}
It is enough to choose
\begin{equation*}
    [u]_{C_x^{\alpha}}|x|^{\alpha} \leq \eta+L_2|x|^2,
\end{equation*}
which holds true if
\begin{equation*}
    L_2 \geq \eta^{\frac{\alpha-2}{\alpha}}[u]_{C_x^{\alpha}}^{\frac{2}{\alpha}}.
\end{equation*}
Thus, by choosing $L_2=8+\eta^{\frac{\alpha-2}{\alpha}}[u]_{C_x^{\alpha}}^{\frac{2}{\alpha}}$, we obtain that \eqref{eq-claim-holder-time} is satisfied on the parabolic boundary $(\partial B_{1/2} \times [t_0, 0]) \cup (\overline B_{1/2} \times \{t_0\})$. 

On the other hand, it is easy to check that a smooth function $\varphi$ is a supersolution of \eqref{eq-parabolic-large} if $L_1 \geq 2\lambda L_2$. Therefore, the definition of viscosity solutions yields that
\begin{equation*}
    \max_{Q_{1/2}} (u-\varphi) =\max_{\partial Q_{1/2}} (u-\varphi) \leq 0,
\end{equation*}
which implies \eqref{eq-claim-holder-time}. If we choose $x=0$, then we conclude that
\begin{equation*}
    u(0, t)-u(0, t_0) \leq \eta+2\lambda \left(8+\eta^{\frac{\alpha-2}{\alpha}}[u]_{C_x^{\alpha}}^{\frac{2}{\alpha}} \right)(t-t_0).
\end{equation*}
A simple optimization with respect to $\eta$ of the right-hand side of the previous inequality shows that, for all $t \in [t_0, 0]$, 
\begin{equation*}
    u(0, t)-u(0, t_0) \leq \tilde C [u]_{C_x^{\alpha}} (t-t_0)^{\frac{\alpha}{2}}
\end{equation*}
for some constant $\tilde C$ depending only on $\lambda$.
\end{proof}

\section{Quantitative H\"older exponent bound in Evans--Krylov theory}\label{sec-EK}
Let $u \in C(B_1)$ be a viscosity solution of
\begin{equation*}
    F(D^2u)=0 \quad \text{in $B_1$},
\end{equation*}
where $F$ is uniformly $(\lambda, \Lambda)$-elliptic and concave. The celebrated Evans--Krylov theory says that $u \in C^{2, \alpha}(\overline B_{1/2})$ for \emph{some} $\alpha=\alpha(n, \lambda, \Lambda) \in (0,1)$. The goal of this section is to suggest a quantitative bound for the H\"older exponent $\alpha$ in terms of $n$ and $\Lambda/\lambda$. The key step is to track down the dependence on $n$, $\lambda$ and $\Lambda$ by employing the Bernstein technique for viscosity solutions $u$ of $F(D^2u)=0$; see \cite[Section~9]{CC95} for instance. 

\begin{remark}[Regularization]\label{rmk-mollification}
     In order to utilize the Bernstein technique, we regularize the operator $F$ by standard mollification as in \cite{LLY24}. To be precise, we extend the domain of $F$ from $\mathcal{S}^n$ to $\mathbb{R}^{n^2}$ by considering $F(M)=F\left(\frac{M+M^T}{2}\right)$. We also let $\psi \in C_c^{\infty}(\mathbb{R}^{n^2})$ be a standard mollifier satisfying $\int_{\mathbb{R}^{n^2}}\psi \,dM=1$ and $\mathrm{supp} \,\psi  \subset \{M \in \mathbb{R}^{n^2} : \sum_{i,j=1}^n M_{ij}^2 \leq 1\}$, and define $\psi_{\varepsilon}(M)=\varepsilon^{-n^2}\psi(M/\varepsilon)$. If we define $F^{\varepsilon}$ as
\begin{equation*}
	F^{\varepsilon}(M) \coloneqq F \ast \psi_{\varepsilon}(M)=\int_{\mathbb{R}^{n^2}} F(M-N) \psi_{\varepsilon}(N)\,dN.
\end{equation*}
It is easy to check that $F^{\varepsilon}$ is uniformly elliptic (with the same ellipticity constants $\lambda, \Lambda$) and smooth. Moreover, $F^{\varepsilon}$ is concave whenever $F$ is concave. Finally, since $F$ is Lipschitz continuous, $F^{\varepsilon}$ converges to $F$ uniformly. Therefore,  we may consider a smooth operator $F^{\varepsilon}$, instead of $F$, in several lemmas of this section.
\end{remark}

We recall the following lemma concerning a linearized operator $L$ of concave and smooth $F$, which plays a crucial role in the Bernstein technique.
\begin{lemma}[{\cite[Lemma~9.2]{CC95}}]\label{lem-linearized}
    Let $F : \mathcal{S}^n \to \mathbb{R}$ be $(\lambda, \Lambda)$-elliptic, concave and smooth with $F(0)=0$. Let $u \in C^4(\Omega)$ satisfy $F(D^2u)=0$ in $\Omega$. Consider the uniformly elliptic operator in $\Omega$
    \begin{equation*}
        Lv = a_{ij}(x)v_{ij} \coloneqq F_{ij}(D^2u(x))v_{ij}.
    \end{equation*}
    Then, for any $e \in \mathbb{R}^n$ with $|e|=1$, 
    \begin{equation*}
        Lu \leq 0, \quad Lu_e=0 \quad \text{and} \quad Lu_{ee} \geq 0 \quad \text{in $\Omega$}.
    \end{equation*}
\end{lemma}

We are now ready to derive uniform gradient and Hessian estimates with the aid of the Bernstein technique.
\begin{proposition}[Uniform gradient and Hessian estimates, {\cite[Proposition~9.3]{CC95}}]\label{prop-bernstein}
    Let $F : \mathcal{S}^n \to \mathbb{R}$ be $(\lambda, \Lambda)$-elliptic, concave and smooth with $F(0)=0$. Let $u \in C^4(\overline B_1)$ satisfy $F(D^2u)=0$ in $B_1$. Then
    \begin{equation*}
        \begin{aligned}
            \|\nabla u\|_{L^{\infty}(B_{1/2})} &\leq C_1\|u\|_{L^{\infty}(B_1)},\\
            \|D^2 u\|_{L^{\infty}(B_{1/2})} &\leq C_2\|\nabla u\|_{L^{\infty}(B_1)},
        \end{aligned}
    \end{equation*}
    where 
    \begin{equation*}
        \begin{aligned}
            C_1=C_1(n, \lambda, \Lambda)&=\sqrt{nC_0 \frac{\Lambda}{\lambda}+4C_0^2 \frac{\Lambda^2}{\lambda^2}},\\
            C_2=C_2(n, \lambda, \Lambda)&=\left(\frac{\Lambda}{\lambda}+1\right)C_1
        \end{aligned}
    \end{equation*}
    for a constant $C_0>0$ independent of $n$, $\lambda$ and $\Lambda$. 
\end{proposition}

\begin{proof}
    We take a cut-off function $\varphi \in C^{\infty}_0(B_1)$ such that $0 \leq \varphi \leq 1$ in $B_1$ and $\varphi \equiv 1$ in $B_{1/2}$. Moreover, 
    \begin{equation*}
         |D_i \varphi|+\|D^2\varphi\|\leq C_{0}=C_{0}(\varphi) \quad \text{in $B_1$ for any $1 \leq i\leq n$}.
    \end{equation*}
    For a particular choice of $\varphi$ and $C_0=C_0(\varphi)$, see Remark~\ref{rmk-C0} below.
    
    For $M \coloneqq \sup_{B_1}u$, we set 
    \begin{equation*}
        h=\delta (M-u)^2+\varphi^2|\nabla u|^2 \in C^3(\overline B_1),
    \end{equation*}
    where $\delta>0$ will be determined soon. In particular, we observe that
    \begin{equation*}
        L(\varphi^2)=2a_{ij}(\varphi \varphi_{ij}+\varphi_i\varphi_j) \geq -2n\|A\|\|D^2\varphi\|+2\lambda |\nabla \varphi|^2 \geq -2n\Lambda C_{0}. 
    \end{equation*}

    We then compute $Lh$, where $L$ is the linearized operator introduced in Lemma~\ref{lem-linearized}. By observing that $L(vw)=(Lv)w+v(Lw)+2a_{ij}v_iw_j$, we have
    \begin{equation*}
        \begin{aligned}
            Lh=2\delta(M-u)(-Lu)&+2\delta a_{ij}u_iu_j+|\nabla u|^2L(\varphi^2)\\
            &+2\varphi^2u_kLu_k+2\varphi^2a_{ij}u_{ki}u_{kj}+8a_{ij}\varphi\varphi_iu_ku_{kj}.
        \end{aligned}
    \end{equation*}
    Since $Lu\leq 0$, $Lu_k=0$ and $(a_{ij})$ has ellipticity $(\lambda, \Lambda)$ from Lemma~\ref{lem-linearized}, we apply the Cauchy--Schwarz inequality to find that
     \begin{equation*}
        \begin{aligned}
            Lh&\geq 2\delta \lambda |\nabla u|^2+|\nabla u|^2L(\varphi^2)+2\lambda \varphi^2\sum_{k, j} (u_{kj})^2+8a_{ij}\varphi\varphi_iu_ku_{kj}\\
            &\geq 2\delta \lambda |\nabla u|^2-2n\Lambda C_{0}|\nabla u|^2+2\lambda \varphi^2\sum_{k, j} (u_{kj})^2-8\Lambda C_{0}\varphi \sum_{k, j}|u_k||u_{kj}|\\
            &\geq 2\delta \lambda |\nabla u|^2-2n\Lambda C_{0}|\nabla u|^2-8C_{0}^2\frac{\Lambda^2}{\lambda}|\nabla u|^2.
        \end{aligned}
    \end{equation*}
    In particular, $Lh \geq 0$ if we choose 
    \begin{equation*}
        \delta=nC_0 \frac{\Lambda}{\lambda}+4C_0^2 \frac{\Lambda^2}{\lambda^2}.
    \end{equation*}
    It then follows from the maximum principle that
    \begin{equation*}
        \sup_{B_{1/2}}|\nabla u|^2 \leq \sup_{B_1}h=\sup_{\partial B_1}h \leq \delta M^2,
    \end{equation*}
    which proves the desired gradient bound.\newline

    We next show the Hessian bound. Consider for $e \in \mathbb{R}^n$ with $|e|=1$ and $\sigma>0$, 
    \begin{equation*}
        g=\sigma (u_e)^2+\varphi^2(u_{ee})^2 \in C^2(\overline \Omega),
    \end{equation*}
    where $\Omega=\{x \in B_1 : u_{ee}(x)>0\}$. We compute $Lg$ in $\Omega$ by a similar argument as before:
    \begin{equation*}
        \begin{aligned}
            Lg&=2\sigma u_eLu_e+2\sigma a_{ij}u_{ei}u_{ej}+(u_{ee})^2L(\varphi^2)\\
            &\qquad+2\varphi^2u_{ee}Lu_{ee}+2\varphi^2a_{ij}u_{eei}u_{eej}+8a_{ij}\varphi \varphi_i u_{ee} u_{eej}\\
            &\geq 2\sigma \lambda \sum_j (u_{ej})^2-2n\Lambda C_0 (u_{ee})^2+2\lambda \varphi^2 \sum_j (u_{eej})^2-8\Lambda C_0\varphi u_{ee}\sum_j u_{eej}\\
            &\geq 2\sigma \lambda \sum_j (u_{ej})^2-2n\Lambda C_0 (u_{ee})^2-8C_0^2\frac{\Lambda^2}{\lambda} (u_{ee})^2
        \end{aligned}
    \end{equation*}
    Hence, $Lg \geq 0$ if we choose 
    \begin{equation*}
        \sigma=nC_0 \frac{\Lambda}{\lambda}+4C_0^2 \frac{\Lambda^2}{\lambda^2}.
    \end{equation*}
    We also note that $\partial \Omega \subset \partial B_1 \cup \{x \in B_1 : u_{ee}(x)=0\}$. It then follows from the maximum principle that
    \begin{equation*}
        \sup_{B_{1/2}}(u_{ee})_+^2=\sup_{\Omega \cap B_{1/2}}(u_{ee})^2 \leq \sup_{\Omega}g=\sup_{\partial \Omega}g \leq \sigma \sup_{B_1} (u_e)^2 \leq \sigma \|\nabla u\|^2_{L^{\infty}(B_1)}.
    \end{equation*}
    In other words, we arrive at
    \begin{equation*}
        \|(u_{ee})_+\|_{L^{\infty}(B_{1/2})} \leq \sqrt{\sigma} \|\nabla u\|_{L^{\infty}(B_1)}
    \end{equation*}
    and \cite[Lemma~6.4]{CC95} gives that
    \begin{equation*}
        \frac{\lambda}{\Lambda+\lambda}\|D^2u\|_{L^{\infty}(B_{1/2})} \leq \sqrt{\sigma} \|\nabla u\|_{L^{\infty}(B_1)}
    \end{equation*}
    as desired.
\end{proof}

We provide a few remarks on Proposition~\ref{prop-bernstein} as follows.
\begin{remark}[Particular choice of a cut-off function $\varphi$]\label{rmk-C0}
   We let
   \begin{equation*}
       \chi(s)=
       \begin{cases}
           0, & s \leq -1,\\
           \frac{1}{\int_{-1}^1 \exp(-1/(1-t^2)) \,\mathrm{d}t}\int_{-1}^s \exp(-1/(1-t^2))\,\mathrm{d}t, & -1<s<1,\\
           1, & s \geq 1.
       \end{cases}
   \end{equation*}
   Then we set
   \begin{equation*}
       s(r) \coloneqq 4r-3 \quad \text{(so that $s(1/2)=-1$ and $s(1)=1$)}
   \end{equation*}
   and
   \begin{equation*}
       \varphi(x) \coloneqq 1-\chi(s(|x|)).
   \end{equation*}
   It is immediate to check that such $\varphi$ satisfies all the conditions given in the proof of Proposition~\ref{prop-bernstein}. Moreover, it follows from a direct calculation that
   \begin{equation*}
       |D_i\varphi|+\|D^2\varphi\| \leq 33 \quad \text{for any $1 \leq i \leq n$.}
   \end{equation*}
   In conclusion, we may set $C_0=33$ with this choice of a cut-off function.
\end{remark}

\begin{remark}[Scaled and combined version]\label{rmk-scaling}
    For the later use, we need a scaled and combined version of estimates proposed in Proposition~\ref{prop-bernstein}. In fact, we observe that $u_r(x)=u(rx+x_0)$ is a solution of $F_r(D^2u_r)=0$ in $B_1$, where $F_r(M) \coloneqq r^2F(M/r^2)$ whenever $B_{r}(x_0) \subset B_1$. By applying the estimates from Proposition~\ref{prop-bernstein} for $u_r$, we have
     \begin{equation*}
        \begin{aligned}
            \|\nabla u\|_{L^{\infty}(B_{r/2}(x_0))} &\leq \frac{C_1}{r}\|u\|_{L^{\infty}(B_r(x_0))},\\
            \|D^2 u\|_{L^{\infty}(B_{r/2}(x_0))} &\leq \frac{C_2}{r}\|\nabla u\|_{L^{\infty}(B_r(x_0))}.
        \end{aligned}
    \end{equation*}
    For any $x_0 \in B_{3/4}$, set $r=1/4$ so that $B_r(x_0) \subset B_1$. Then we arrive at
    \begin{equation*}
         \|D^2 u\|_{L^{\infty}(B_{r/4}(x_0))} \leq \frac{2C_2}{r}\|\nabla u\|_{L^{\infty}(B_{r/2}(x_0))} \leq \frac{2C_1C_2}{r^2}\|u\|_{L^{\infty}(B_r(x_0))},
    \end{equation*}
    and so
    \begin{equation*}
        \|D^2 u\|_{L^{\infty}(B_{3/4})}  \leq 32C_1C_2\|u\|_{L^{\infty}(B_1)}.
    \end{equation*}
\end{remark}

\begin{remark}[Rough bound for $C_1C_2$]\label{rmk-rough}
    Since we will let $\Lambda/\lambda$ be sufficiently close to $1$, we may assume that $\Lambda/\lambda \leq 2$ in the remainder of this section. Then we have a rather simple bound for $C_1C_2$:
    \begin{equation*}
        C_1C_2=\left(\frac{\Lambda}{\lambda}+1\right) \left( nC_0 \frac{\Lambda}{\lambda}+4C_0^2 \frac{\Lambda^2}{\lambda^2} \right) \leq 6(nC_0+8C_0^2) \, (\sim n).
    \end{equation*}
\end{remark}

\begin{lemma}[Approximation lemma]\label{lem-approximation}
Let $F : \mathcal{S}^n \to \mathbb{R}$ be $(\lambda, \Lambda)$-elliptic, concave and smooth with $F(0)=0$. Moreover, let $C_1$, $C_2$ be the constants chosen in Proposition~\ref{prop-bernstein}, and let $u \in C(B_1)$ be a viscosity solution of
\begin{equation*}
    F(D^2u)=0 \quad \text{in $B_1$}.
\end{equation*}
Given $\varepsilon \in (0,1/12)$, if
\begin{equation*}
    \frac{\Lambda}{\lambda}-1 \leq \frac{\varepsilon}{9C_1C_2},
\end{equation*}
then there exists a harmonic function $h \in C^2(\overline B_{3/4})$ such that
\begin{equation*}
    \|u-h\|_{L^{\infty}(B_{3/4})} \leq \varepsilon\|u\|_{L^{\infty}(B_1)}.
\end{equation*}
\end{lemma}

\begin{proof}
By replacing $u$ with $u/\|u\|_{L^\infty(B_1)}$, we may assume that $\|u\|_{L^\infty(B_1)}\leq 1$. 

Let $M\in\mathcal{S}^n$ be a symmetric matrix with eigenvalues $\{e_i\}_{i=1}^n$. Then we observe that
\begin{equation}\label{eq:F-harmonic}
    \begin{aligned}
        |F(M)-\lambda \mathrm{tr} M|
&\leq \max\{\mathcal M_{\lambda , \Lambda}^+(M)-\lambda \mathrm{tr} M, \ \lambda \mathrm{tr} M-\mathcal M_{\lambda, \Lambda}^-(M)\}\\
&\leq  (\Lambda-\lambda)\left(\sum_{e_i>0}e_i+\sum_{e_i<0}(-e_i)\right)
 \leq (\Lambda-\lambda) n \|M\|.
    \end{aligned}
\end{equation}

We now let $h$ be the harmonic replacement of $u$ in $B_{3/4}$, i.e.,
\begin{equation*}
\left\{
    \begin{aligned}
        \Delta h&=0 && \text{in $B_{3/4}$}\\
        h&=u && \text{on $\partial B_{3/4}$}.
    \end{aligned}
    \right.
\end{equation*}
For $w \coloneqq u-h$ in $B_{3/4}$, we observe that
\begin{equation*}
    \Delta w=\Delta u-\Delta h=\Delta u,
\end{equation*}
and \eqref{eq:F-harmonic} implies that
\begin{equation*}
    |\Delta u| \leq \left(\frac{\Lambda}{\lambda}-1\right) n \|D^2u\|.
\end{equation*}
We next let $N:=\|\Delta u\|_{L^\infty(B_{3/4})}$ and consider the barrier
\begin{equation*}
    \phi(x)\coloneqq \frac{N}{2n}((3/4)^2-|x|^2)\quad\text{in }B_{3/4}.
\end{equation*}
Since $\Delta \phi=-N$, the comparison principle between $\pm w$ and $\phi$ in $B_{3/4}$ yields that
\begin{equation*}
    |w| \leq \phi \quad \text{in $B_{3/4}$} 
\end{equation*}
and so 
\begin{equation*}
    \|w\|_{L^{\infty} (B_{3/4})} \leq \frac{N}{2n} \left(\frac{3}{4}\right)^2.
\end{equation*}
Then we recall \eqref{eq:F-harmonic} and apply a scaled version of the interior $C^2$ estimate for $u$ (Remark~\ref{rmk-scaling}):
\begin{equation*}
    N \leq \left(\frac{\Lambda}{\lambda}-1\right) n \|D^2u\|_{L^{\infty}(B_{3/4})} \leq 32\left(\frac{\Lambda}{\lambda}-1\right)n C_1 C_2.
\end{equation*}
We finally arrive at
\begin{equation}\label{eq:w-Linfty}
    \|u-h\|_{L^\infty(B_{3/4})} \leq  9\left(\frac{\Lambda}{\lambda}-1\right) C_1 C_2 \leq \varepsilon,
\end{equation}
under the assumption on the ellipticity ratio.
\end{proof}

We now prove the following theorem, which is corresponding to Theorem~\ref{thm-c2alpha-general} when $F$ is smooth.
\begin{theorem}[$C^{2,\alpha}$ estimate; smooth $F$]\label{thm-c2alpha-smooth}
   Let $F : \mathcal{S}^n \to \mathbb{R}$ be $(\lambda, \Lambda)$-elliptic, concave and smooth with $F(0)=0$. Moreover, $C_1$, $C_2$ be the constants chosen in Proposition~\ref{prop-bernstein}, and let $u \in C(B_1)$ be a viscosity solution of
\begin{equation*}
    F(D^2u)=0 \quad \text{in $B_1$}.
\end{equation*}
Given $\alpha \in (0,1)$, if
\begin{equation*}
    \frac{\Lambda}{\lambda}-1 \leq \frac{1}{18C_1 C_2} (576n^3)^{-\frac{2+\alpha}{1-\alpha}},
\end{equation*}
then $u\in C^{2,\alpha}(\overline B_{1/2})$ with the uniform estimate
\begin{equation*}
    \|u\|_{C^{2,\alpha}(\overline B_{1/2})} \leq C\|u\|_{L^\infty(B_1)},
\end{equation*}
where $C>0$ depends only on $n$, $\lambda$, $\Lambda$ and $\alpha$.
\end{theorem}

\begin{proof}
As in the proof of Lemma~\ref{lem-approximation}, we may assume that $\|u\|_{L^{\infty}(B_1)}\leq 1$. By applying Lemma~\ref{lem-approximation}, there exists a harmonic function $h \in C^2(\overline B_{3/4})$ such that
\begin{equation}\label{eq-approxlemma}
    \|u-h\|_{L^{\infty}(B_{3/4})} \leq \varepsilon,
\end{equation}
provided that 
\begin{equation*}
    \frac{\Lambda}{\lambda}-1 \leq \frac{\varepsilon}{9C_1 C_2} \quad \text{for $\varepsilon>0$ will be determined soon}.
\end{equation*}
 We note that $\|h\|_{L^{\infty}(B_{3/4})} \leq \|u\|_{L^{\infty}(B_1)} \leq 1$ by its construction. In particular, $h$ has a quadratic Taylor polynomial $P$ at the origin, which satisfies $\Delta P=0$ and 
\begin{equation*}
    |P|=\left|h(0)+Dh(0) \cdot x+\frac{1}{2}x^T D^2h(0)x\right| \leq 1+3n+18n^2 \leq c_0 \quad \text{in $B_{3/4}$}.
\end{equation*}
Moreover, since $h$ is harmonic, we have
\begin{equation}\label{eq-taylor}
    \|h-P\|_{L^{\infty}(B_r)} \leq \frac{1}{6}\|D^3h\|_{L^{\infty}(B_{3/4})}r^3 \leq 288n^3r^3 \eqqcolon c_0r^3 \quad \text{for all $r \leq 3/4$}.
\end{equation}
Here we used the interior estimates for a harmonic function $h$ (Lemma~\ref{lem-est-harmonic}).

We combine \eqref{eq-approxlemma} and \eqref{eq-taylor} to obtain that
\begin{equation*}
    \|u-P\|_{L^{\infty}(B_r)} \leq c_0r^3+\varepsilon \quad \text{for all $r \leq 3/4$}.
\end{equation*}
Choose now $r_0$ small enough so that $c_0r_0^3 \leq r_0^{2+\alpha}/2$ (notice $\alpha<1$), and then $\varepsilon$ small enough so that $\varepsilon\leq r_0^{2+\alpha}/2$. Then we obtain
\begin{equation*}
    \|u-P\|_{L^{\infty}(B_{r_0})} \leq r_0^{2+\alpha}.
\end{equation*}
The remaining iteration argument is quite standard; see \cite[Section~8.1]{CC95} for instance.
\end{proof}

\begin{remark}[The condition on the ellipticity ratio]\label{rmk-constant-A}
    Since the condition on $\Lambda/\lambda$ given in Theorem~\ref{thm-c2alpha-smooth} is rather implicit, we may recall the observation from Remark~\ref{rmk-rough}. In fact, since $C_1C_2 \leq 6(nC_0+8C_0^2)$, the condition on the ellipticity ratio will be satisfied when we assume
    \begin{equation*}
        \frac{\Lambda}{\lambda}-1 \leq \frac{1}{108(nC_0+8C_0^2)} (576n^3)^{-\frac{2+\alpha}{1-\alpha}}.
    \end{equation*}
    We also note that one candidate of $C_0$ is $33$ from Remark~\ref{rmk-C0}.
    
    On the other hand, in a more heuristic manner, we have
    \begin{equation*}
        \frac{\Lambda}{\lambda}-1  \approx \frac{1}{n} \left(\frac{1}{n^3} \right)^{\frac{2+\alpha}{1-\alpha}} \quad \text{up to a multiplicative constant}.
    \end{equation*}
    In particular, it is clear that if $\alpha \to 1^-$, then $\Lambda/\lambda  \to 1^+$ as expected.
\end{remark}

We finally employ the mollification method to derive the same estimate without the smoothness assumption on $F$.

\begin{proof}[Proof of Theorem~\ref{thm-c2alpha-general}]
    As in the proof of Lemma~\ref{lem-approximation}, we may assume that $\|u\|_{L^{\infty}(B_2)}\leq 1$. Moreover, we consider the mollification $F^{\varepsilon}$ of $F$ as in Remark~\ref{rmk-mollification} and let $u^{\varepsilon} \in C^{\infty}(B_1) \cap C(\overline B_1)$ be a unique solution of
    \begin{equation*}
        \left\{
        \begin{aligned}
            F^{\varepsilon}(D^2u^{\varepsilon})&=0  &&\text{in $B_1$},\\
            u^{\varepsilon}&=u &&\text{on $\partial B_1$}.
        \end{aligned}
        \right.
    \end{equation*}
    We now apply Theorem~\ref{thm-c2alpha-smooth} for $u^{\varepsilon}$ to obtain
    \begin{equation*}
    \|u^{\varepsilon}\|_{C^{2,\alpha}(\overline B_{1/2})} \leq C\|u^{\varepsilon}\|_{L^\infty(B_1)} \leq C\|u\|_{L^{\infty}(B_2)},
\end{equation*}
where $C>0$ is independent of $\varepsilon$. 

On the other hand, it is easy to check that there exists a subsequence of $\{u^{\varepsilon}\}_{\varepsilon>0}$ that converges uniformly (in $C^2$) to some function $v$ in every compact subset of $B_1$. Then the stability theorem (Theorem~\ref{thm-stability}) implies that $v$ becomes a viscosity solution of $F(D^2v)=0$ in $B_1$. Furthermore,  we also have $v=u$ on $\partial B_1$ in view of the global H\"older estimate for $u^{\varepsilon}$ (Theorem~\ref{thm-global-holder}). Thus, we conclude that $v=u$ and in particular, 
\begin{equation*}
    \|u\|_{C^{2,\alpha}(\overline B_{1/2})}  \leq C\|u\|_{L^{\infty}(B_2)}
\end{equation*}
as desired.
\end{proof}


\end{document}